\documentclass[11pt]{amsart}
\usepackage{amsmath,amssymb,amsthm,enumitem,hyperref,mathtools}
\usepackage[margin=1.05in]{geometry}
\usepackage[T1]{fontenc}
\usepackage{lmodern}
\usepackage{microtype}
\usepackage{longtable}
\usepackage{booktabs}
\usepackage{array}
\hypersetup{colorlinks=true, linkcolor=blue, citecolor=blue, urlcolor=blue}

\newtheorem{theorem}{Theorem}[section]
\newtheorem{lemma}[theorem]{Lemma}

\theoremstyle{definition}
\newtheorem{definition}[theorem]{Definition}
\newtheorem{example}[theorem]{Example}
\theoremstyle{remark}
\newtheorem{remark}[theorem]{Remark}

\newcommand{\D}{\mathcal D}
\newcommand{\Oparts}{\mathcal O}
\newcommand{\Aone}{\mathcal A_1}
\newcommand{\Atwo}{\mathcal A_2}
\newcommand{\mult}{\operatorname{mult}}
\newcommand{\length}{\ell}
\newcommand{\modulo}[1]{\;(\mathrm{mod}\ #1)}

\title[A bijective proof of a partition theorem of Berkovich and Uncu]{A bijective proof of a partition theorem of Berkovich and Uncu}

\author{Michal Mogielnicki, Ken Ono, Niels Voss, and Jujian Zhang}
\address{Axiom Math, 124 University Avenue, Palo Alto, CA 94301}
\email{michal@axiommath.ai}
\email{ken@axiommath.ai}
\email{niels@axiommath.ai}

\begin{document}
\begin{abstract}
In 2016, Berkovich and Uncu proved that, for all nonnegative integers $i$, $j$, and $n$, the number of strict partitions of $n$ with $i$ odd-indexed odd parts and $j$ even-indexed odd parts equals the number of strict partitions of $n$ with $i$ parts congruent to $1$ modulo $4$ and $j$ parts congruent to $3$ modulo $4$. Their proof used generating functions, and they asked for a combinatorial proof. We answer their question with an explicit bijection, assembled from three classical ingredients: $2$-modular diagrams, an insertion algorithm of Chen, Gao, Ji, and Li, and Glaisher's bijection.
AxiomProver autonomously formalized and verified  the proof of the main theorem in Lean.
\end{abstract}



\maketitle

\section{Introduction}\label{sec:intro}
We begin by fixing notation. A \emph{partition} $\lambda$ of $n$, written $\lambda\vdash n$, is a weakly decreasing sequence of positive integers
\[
  \lambda=(\lambda_1\ge\lambda_2\ge\cdots\ge\lambda_\ell),
  \qquad \lambda_1+\cdots+\lambda_\ell=n,
\]
and we write \(\length(\lambda)=\ell\) for its number of parts. For a positive integer $m$, we write $\mult_\lambda(m)$ for the number of parts of $\lambda$ equal to $m$.
A partition $\lambda \vdash n$ is \emph{strict} if its parts are distinct, i.e.\ a strict
partition of \(n\) has the form
\[
  \lambda=(\lambda_1>\lambda_2>\cdots>\lambda_\ell),
  \qquad \lambda_1+\cdots+\lambda_\ell=n.
\]
For such a strict partition define
\[
  e_1(\lambda)=\#\{1\le r\le \length(\lambda): \lambda_r\equiv 1\modulo 2,\ r\equiv 1\modulo 2\},
\]
and
\[
  e_3(\lambda)=\#\{1\le r\le \length(\lambda): \lambda_r\equiv 1\modulo 2,\ r\equiv 0\modulo 2\}.
\]
Thus \(e_1(\lambda)\) counts the odd parts occurring in an odd-indexed position,
while \(e_3(\lambda)\) counts the odd parts occurring in an even-indexed position.
For an arbitrary partition \(\lambda\) and \(a\in\{1,3\}\) put
\[
  r_a(\lambda)=\#\{1\le s\le \length(\lambda): \lambda_s\equiv a\modulo 4\}.
\]
The subscripts $1$ and $3$ on $e_1$ and $e_3$ anticipate Lemma~\ref{lem:phi} below: under the first of our four maps, odd-indexed odd parts become parts congruent to $1$ modulo $4$, and even-indexed odd parts become parts congruent to $3$ modulo $4$.

The following elegant identity is the subject of this paper.

\begin{theorem}[Berkovich--Uncu {\cite[Theorem 1.1]{BU}}]\label{thm:BU}
For nonnegative integers \(i,j,n\), the number of strict partitions \(\lambda\vdash n\)
with
\[
   e_1(\lambda)=i,\qquad e_3(\lambda)=j
\]
equals the number of strict partitions \(\sigma\vdash n\) with
\[
   r_1(\sigma)=i,\qquad r_3(\sigma)=j.
\]
\end{theorem}

Berkovich and Uncu proved Theorem~\ref{thm:BU} by generating functions and noted
that it calls for a combinatorial proof \cite[Section 7]{BU}. We give such a proof, in the strongest sense: an explicit bijection, computable in both directions.

\begin{remark}\label{rem:FZ}
Theorem~\ref{thm:BU} has received combinatorial attention before. Fu and Zeng \cite[Theorem 1.2 and Section 4.1]{FZ} obtain it by dissecting labelled Ferrers diagrams to derive a bounded form of Boulet's four-parameter generating function, and then extracting coefficients. Their argument is combinatorial, but it passes through a cancellation of generating functions and yields an equality of coefficients rather than a map: no correspondence between the two families of partitions results. The contribution of the present paper is such a correspondence. Separately, Dhar and Mukhopadhyay \cite{DM} have given a bijective proof of a different identity of Berkovich and Uncu, namely the $q$-binomial formula \cite[Theorem 3.1]{BU} for the generating function of strict partitions with fixed BG-rank; the BG-rank is the \emph{difference} $e_1-e_3$, whereas Theorem~\ref{thm:BU} concerns the joint distribution of the pair $(e_1,e_3)$.
\end{remark}

Our bijection is a composition of four maps, each of which refines Euler's classical theorem that the number of strict partitions of $n$ equals the number of partitions of $n$ into odd parts. Beyond Euler's theorem, the paper is self-contained except for bijectivity statements quoted from Chen--Gao--Ji--Li \cite{CGJL} and Glaisher \cite{Glaisher}; no generating functions are used. After fixing the intermediate classes of partitions in
Section~\ref{sec:classes}, we introduce the four maps in Sections~\ref{sec:map1}--\ref{sec:map4},
each with a worked example and an accompanying lemma describing precisely which statistic it transports.
The maps are assembled in Section~\ref{sec:proof}, where Theorem~\ref{thm:BU} is
deduced.  Section~\ref{sec:numerics} illustrates the bijection: a single partition is traced through all four maps, and a table exhibits the bijection on an entire fiber. The proof of this result was autonomously formalized and verified in Lean by AxiomProver, a tool for mathematical research that is currently under development. Section~\ref{sec:AI} describes the accompanying Lean/mathlib formalization, including the protocol, and the input and output files. 

\section{The proof of the theorem}\label{sec:mainproof}

This section carries out the construction. Section~\ref{sec:classes} introduces the two intermediate classes of partitions through which the bijection travels and explains the strategy; Sections~\ref{sec:map1}--\ref{sec:map4} define the four maps; Section~\ref{sec:proof} composes them and proves the theorem.

\subsection{Intermediate classes}\label{sec:classes}
Let $\D(n)$ be the set of strict partitions of $n$ and $\Oparts(n)$ be the set of partitions of $n$ into odd parts. Following Chen--Gao--Ji--Li \cite[Section 2]{CGJL}, we define \(\Aone(n)\) to be the set of partitions \(\alpha=(\alpha_1\ge\cdots\ge\alpha_s)\vdash n\) satisfying:
\begin{enumerate}[label=\textup{(A1.\arabic*)}]
\item only even parts may be repeated;
\item \(\alpha_r-\alpha_{r+1}\le 4\) for all \(1\le r<s\);
\item \(\alpha_r-\alpha_{r+1}<4\) if $\alpha_r$ or $\alpha_{r+1}$ is even, for all \(1\le r<s\);
\item the smallest part satisfies \(\alpha_s<4\) provided \(\alpha\) is nonempty,
\end{enumerate}
and \(\Atwo(n)\) be the set of partitions \(\beta\vdash n\) satisfying:
\begin{enumerate}[label=\textup{(A2.\arabic*)}]
\item no part is divisible by \(4\);
\item only even parts may be repeated.
\end{enumerate}
In any \(\beta\in\Atwo(n)\), every odd part occurs at most once,
and every part is congruent to \(1\), \(2\), or \(3\) modulo \(4\). Therefore,
\(r_1(\beta)\) and \(r_3(\beta)\) count only odd parts of \(\beta\).

\begin{example}\label{ex:classes}
The partition $(6,3,2,2,2,1)$ lies in $\Aone(16)$ and in $\Atwo(16)$. By contrast:
$(3,3,1)$ fails \textup{(A1.1)} (a repeated odd part);
$(7,2)$ fails \textup{(A1.2)} (a gap of $5$);
$(6,2,1)$ fails \textup{(A1.3)} (a gap of exactly $4$ between even parts);
$(5,4)$ fails \textup{(A1.4)} (smallest part $4$);
$(8,3)$ fails \textup{(A2.1)}; and $(3,3,2)$ fails \textup{(A2.2)}.
The classes $\Aone(n)$ and $\Atwo(n)$ are the $N=2$ case of families introduced by Chen--Gao--Ji--Li \cite{CGJL} for arbitrary modulus $N$.
\end{example}

Our bijection is the composition
\[
   F=G\circ \psi^{-1}\circ \Phi\circ \varphi:
   \D(n)\longrightarrow \D(n),
\]
where
\[
   \D(n)\xrightarrow{\ \varphi\ }\Aone(n)
        \xrightarrow{\ \Phi\ }\Atwo(n)
        \xrightarrow{\ \psi^{-1}\ }\Oparts(n)
        \xrightarrow{\ G\ }\D(n).
\]
Throughout the paper we reserve the letters $\alpha=\varphi(\lambda)$, $\beta=\Phi(\alpha)$, $\mu=\psi^{-1}(\beta)$, and $\sigma=G(\mu)=F(\lambda)$ for the successive images of a strict partition $\lambda$.

Before defining the maps we record why all four classes have the same size, and sketch the strategy.

\begin{remark}[All four classes are equinumerous]\label{rem:euler}
Euler's theorem gives $|\D(n)|=|\Oparts(n)|$. The maps quoted below give $|\D(n)|=|\Aone(n)|$ \cite[Lemma 2.2]{CGJL} and $|\Oparts(n)|=|\Atwo(n)|$ \cite[Lemma 2.3]{CGJL}. Each of the four maps is therefore a bijection between equinumerous refinements of Euler's theorem; the substance of this paper is the bookkeeping of the statistics $(e_1,e_3)$ and $(r_1,r_3)$ along the way.
\end{remark}

The reader may keep the following picture in mind. The first map $\varphi$ converts a \emph{position} statistic into a \emph{residue} statistic: an odd part of $\lambda$ in position $r$ becomes a part of $\alpha$ equal to $2r-1$, which is congruent to $1$ modulo $4$ when $r$ is odd and to $3$ modulo $4$ when $r$ is even (Lemma~\ref{lem:phi}). The second map $\Phi$ changes parts only by multiples of $4$, so it preserves the pair $(r_1,r_3)$ outright (Lemma~\ref{lem:Phi}). The last two maps trade multiplicities for binary expansions, and in doing so they preserve, value for value, the \emph{odd} parts (Lemmas~\ref{lem:psiinv} and~\ref{lem:G}). The pair of statistics therefore survives the entire journey.

\subsection{Map 1: \texorpdfstring{$2$}{2}-modular conjugation}\label{sec:map1}

The first map re-encodes a strict partition through its $2$-modular diagram (also called its $2$-modular Ferrers diagram, or MacMahon diagram), turning the parity of a part's \emph{position} into a residue class modulo $4$.

\begin{definition}[The map \(\varphi:\D(n)\to\Aone(n)\)]\label{def:phi}
Let $\lambda = (\lambda_1>\cdots>\lambda_\ell)\in\D(n)$. The \emph{$2$-modular diagram} of $\lambda$ has one row per part: each part is written as a row of $2$'s followed by a single $1$ if the part is odd. Namely, for $q\in\mathbb{N}=\{0,1,2,\dots\}$:
$$2q\ \longmapsto\ \underbrace{2\ 2\ \cdots\ 2}_{q},
  \qquad
  2q+1\ \longmapsto\ \underbrace{2\ 2\ \cdots\ 2}_{q}\ 1 .
$$
Then \(\varphi(\lambda)\) is the sequence of column sums of this diagram, read left to right. Since the entries of each row weakly decrease along the row, the column sums weakly decrease, so $\varphi(\lambda)$ is a partition of $n$; equivalently, $\varphi(\lambda)$ lists the row sums of the transposed (conjugated) diagram, whence the name of the map. That $\varphi(\lambda)$ lies in $\Aone(n)$, and that $\varphi$ is a bijection onto $\Aone(n)$, is \cite[Lemma 2.2]{CGJL}.
\end{definition}

\begin{example}\label{ex:phi}
Take $\lambda=(5,3,2)$, with odd parts $5$ and $3$ in positions $1$ and $2$, so $(e_1,e_3)=(1,1)$. The $2$-modular diagram and its transpose are
\[
\begin{array}{@{}ccc@{}}
2 & 2 & 1\\
2 & 1 & \\
2 & &
\end{array}
\qquad\longrightarrow\qquad
\begin{array}{@{}ccc@{}}
2 & 2 & 2\\
2 & 1 & \\
1 & &
\end{array}
\]
and the column sums of the left diagram (the row sums of the right one) give
\[
   \varphi(5,3,2)=(2+2+2,\ 2+1,\ 1)=(6,3,1).
\]
Note that $(6,3,1)$ has one part congruent to $1$ and one part congruent to $3$ modulo $4$, matching $(e_1,e_3)=(1,1)$; the next lemma shows this is no accident.
\end{example}

The following lemma is the engine of the whole proof: it converts the position statistic $(e_1,e_3)$ into the residue statistic $(r_1,r_3)$.

\begin{lemma}[Statistic transfer under \(\varphi\)]\label{lem:phi}
For all $\lambda\in\D(n)$, $r_1(\varphi(\lambda))=e_1(\lambda)$ and $r_3(\varphi(\lambda))=e_3(\lambda)$.
\end{lemma}

\begin{proof}
The columns of the \(2\)-modular diagram containing a terminal \(1\) are in
bijection with the odd parts of \(\lambda\); a column without a terminal \(1\) has an even sum, and a column with a terminal \(1\) has an odd sum. We compute these odd
sums. Suppose \(\lambda_r\) is odd, say \(\lambda_r=2c-1\). The row $r$ then consists of \(c-1\) entries equal to \(2\) followed by a terminal \(1\) in column \(c\). As \(\lambda\) is strict, every row above row \(r\) has size at least \(\lambda_r+1=2c\), hence contains at least \(c\) entries; its entry in column $c$ cannot be a terminal $1$, since a terminal $1$ in column $c$ belongs to a row of size $2c-1<2c$, so that entry is a $2$. Every row below row \(r\) has size at most \(\lambda_r-1=2c-2\), hence contains at most \(c-1\) entries and contributes nothing in column \(c\).
The sum of column \(c\) is therefore
\[
   2(r-1)+1=2r-1.
\]
If \(r\) is odd, \(2r-1\equiv 1\modulo 4\); if \(r\) is even,
\(2r-1\equiv 3\modulo 4\). Conversely, every odd part of \(\varphi(\lambda)\) is such a column sum and arises from a unique terminal \(1\)---a unique odd part of \(\lambda\). The parts of \(\varphi(\lambda)\) that are
congruent to \(1\modulo 4\) correspond exactly to the odd-indexed odd parts of
\(\lambda\), and those congruent to \(3\modulo 4\) to the even-indexed odd parts.
\end{proof}

\subsection{Map 2: the Chen--Gao--Ji--Li insertion map}\label{sec:map2}
The second map clears out the multiples of $4$. It is the \(N=2\) specialization of the Chen--Gao--Ji--Li insertion bijection \cite[Theorem 2.4 and Section 3]{CGJL}, which refines Bessenrodt's insertion algorithm \cite{Bess1991,Bess1995}; it is a bijection \(\Aone(n)\to\Atwo(n)\) by \cite[Section 3]{CGJL}. We restate the specialized algorithm, since we track its operations in the proof of Lemma~\ref{lem:Phi}. Phases (i) and (ii) below are Steps 1 and 2 of \cite[Section 3]{CGJL}, specialized to $N=2$; ``removable'' is our shorthand for their deletion criterion.

\begin{definition}[The map \(\Phi:\Aone(n)\to\Atwo(n)\)]\label{def:Phi}
Let \(\alpha\in\Aone(n)\). Throughout, the current object is a partition, written
in weakly decreasing order; ``reorder'' means re-sort into weakly decreasing
order. The algorithm has an extraction stage, producing a pair
\((\alpha^\ast,\delta)\), where \(\alpha^\ast\) is a partition of \(n-|\delta|\) satisfying \textup{(A1.1)}--\textup{(A1.4)} and \textup{(A2.1)}--\textup{(A2.2)}, and \(\delta\)
is a partition into positive multiples of \(4\), followed by an insertion stage.

\smallskip
\noindent\emph{Extraction.}
\begin{enumerate}[label=\textup{(\roman*)}]
\item \emph{Pre-extraction.} Repeat the following step until it no longer applies.
   Call a part \(\alpha_j\) of the current partition \emph{removable} if it is divisible by \(4\) and one of the following holds:
   \begin{itemize}
   \item \(\alpha_j\) is the largest part; or
   \item \(\alpha_j\) is an interior part and the pair
         \((\alpha_{j-1},\alpha_{j+1})\) of its neighbours satisfies the
         \(\Aone\) gap condition, namely \(\alpha_{j-1}-\alpha_{j+1}\le 4\), with strict
         inequality if either of \(\alpha_{j-1},\alpha_{j+1}\) is even.\footnote{The general algorithm of \cite[Section 3]{CGJL} also permits removing a \emph{smallest} part divisible by $4$. By Remark~\ref{rem:wd}(a), condition \textup{(A1.4)} keeps the smallest part below $4$ throughout this phase, so that case never arises for $N=2$ and we omit it.}
   \end{itemize}
   If several parts are removable, remove the largest of them (the leftmost such part if that value is repeated); this fixed rule makes the procedure deterministic. Record the removed value, unchanged, as a part of
   \(\delta\). (The gap test is what selects which multiples of \(4\) may be freed;
   ``removable'' is not a synonym for ``divisible by~\(4\).'')
\item \emph{Iterative extraction.} While the current partition still contains a
   multiple of \(4\), choose the \emph{largest} such part; suppose it occupies
   position \(s\) (counting from the largest part) and equals \(4m\). Subtract \(4\)
   from each of the \(s-1\) parts preceding it, delete the part \(4m\), reorder, and
   record as a part of \(\delta\) the multiple of \(4\)
   \[
      4(s-1)+4m .
   \]
   Repeat until no multiple of \(4\) remains. The partition that remains is
   \(\alpha^\ast\).
\end{enumerate}

\noindent\emph{Insertion.} Insert the recorded parts of \(\delta\) in weakly decreasing
order. To insert a part \(4h\), add \(4\) to each of the first \(h\) parts of the
current partition and then reorder. After all parts of \(\delta\) have been
inserted, the resulting partition is \(\Phi(\alpha)\).

\smallskip
The inverse \(\Phi^{-1}:\Atwo(n)\to\Aone(n)\) reverses the two stages: the
insertion steps are undone in weakly increasing order of the inserted parts, and the two
extraction phases are run in reverse. We do not reprove invertibility here; see \cite[Section 3]{CGJL} for the inverse
algorithm and the proof that \(\Phi\) is a bijection.
\end{definition}

Before giving examples we verify that every step of the algorithm is meaningful.

\begin{remark}[Well-definedness]\label{rem:wd}
Three observations keep the algorithm well defined.
\begin{enumerate}[label=\textup{(\alph*)}]
\item Throughout phase (i), the current partition satisfies \textup{(A1.1)}--\textup{(A1.4)} at its current weight: each removal deletes either the largest part or an interior part whose surrounding pair passes the gap test, and neither operation can violate the four conditions. In particular the smallest part stays below $4$.
\item If a multiple of $4$ is immediately preceded by an equal part, it is removable: it is not the smallest part (that part is below $4$ by (a)), and deleting it leaves the pair $(\alpha_{j-1},\alpha_{j+1})$ with $\alpha_{j-1}-\alpha_{j+1}=\alpha_j-\alpha_{j+1}<4$, the strict inequality holding by \textup{(A1.3)} since $\alpha_j$ is even. Hence when phase (ii) begins, the multiples of $4$ present are pairwise distinct, and each iteration of phase (ii) preserves this: the multiples of $4$ preceding the chosen part $4m$ exceed it, so they remain at least $4m$ after the subtraction, while those following it are at most $4m-4$, and the subtraction preserves the differences among the preceding parts. Consequently ``the largest multiple of $4$'' occupies a unique position $s$, and every part preceding position $s$ strictly exceeds $4m\ge 4$, so the subtractions leave all parts positive.
\item By \cite[Section 3]{CGJL}, the largest part of $\delta$ is at most $4\length(\alpha^\ast)$. Since inserting a part changes no part count, every intermediate partition in the insertion stage has $\length(\alpha^\ast)$ parts, and each insertion step is defined.
\end{enumerate}
\end{remark}

\begin{example}[Phase (i) only]\label{ex:Phi1}
Take \(\alpha=(6,4,3,2,1)\in\Aone(16)\). The only multiple of $4$ is the part $4$; deleting it leaves the neighbouring pair $(6,3)$, and as $6-3=3<4$, the $\Aone$ gap condition holds and $4$ is removable. We remove it and record $4$, leaving $$(6,3,2,1),\qquad \delta=(4).$$
No multiple of \(4\) remains, so iterative extraction is vacuous and
\(\alpha^\ast=(6,3,2,1)\). We now insert \(\delta=(4)=(4\cdot 1)\), adding \(4\) to
the first part:
\[
   (6,3,2,1)\longmapsto(10,3,2,1).
\]
Hence $$\Phi(6,4,3,2,1)=(10,3,2,1).$$
The residue counts are preserved: both \((6,4,3,2,1)\) and \((10,3,2,1)\) have
exactly one part congruent to \(1\modulo 4\) and one congruent to \(3\modulo 4\).
\end{example}

\begin{example}[Both phases]\label{ex:Phi2}
Take \(\alpha=(7,4,4,1)\), which lies in $\Aone(16)$: the only repeated part is the even part $4$, all gaps are $3,0,3$, and the smallest part is $1$. In pre-extraction (i) both $4$'s are candidates; following our rule we test the leftmost, in position $2$. Deleting it leaves the pair $(7,4)$ with difference $7-4=3<4$, so it is removable. We remove it and record $4$, leaving
\[
   (7,4,1),\qquad \delta=(4).
\]
A multiple of \(4\) remains (position \(2\)), but it is not removable:
deleting it would leave the pair \((7,1)\) with difference \(6>4\). So we
pass to iterative extraction (ii). The largest (and only) multiple of \(4\) is the
part \(4\) in position \(s=2\); we subtract \(4\) from the single preceding part
\(7\), delete the \(4\), and record \(4(2-1)+4=8\). This gives
\[
   \alpha^\ast=(3,1),\qquad \delta=(8,4).
\]
Now insert \(\delta\) in weakly decreasing order. Inserting \(8=4\cdot2\) adds \(4\) to the
first two parts:
\[
   (3,1)\longmapsto(7,5).
\]
Inserting \(4=4\cdot1\) adds \(4\) to the first part:
\[
   (7,5)\longmapsto(11,5).
\]
Hence $\Phi(7,4,4,1)=(11,5)$; again one part lies in each odd residue class modulo $4$ on both sides.
\end{example}

Since every operation of the algorithm shifts parts by multiples of $4$ or handles parts divisible by $4$, the residue statistics cannot change. The next lemma records this.

\begin{lemma}[Residue refinement of \(\Phi\)]\label{lem:Phi}
For every \(\alpha\in\Aone(n)\),
\[
   r_1(\Phi(\alpha))=r_1(\alpha),\qquad
   r_3(\Phi(\alpha))=r_3(\alpha).
\]
Consequently, for all \(n,i,j\ge 0\), \(\Phi\) restricts to a bijection
\[
   \{\alpha\in\Aone(n):r_1(\alpha)=i,\ r_3(\alpha)=j\}
   \longrightarrow
   \{\beta\in\Atwo(n):r_1(\beta)=i,\ r_3(\beta)=j\}.
\]
\end{lemma}

\begin{proof}
We track the effect of each elementary operation of Definition~\ref{def:Phi} on
the residues of parts modulo \(4\). The quantities \(r_1\) and \(r_3\) count parts
in the residue classes \(1\) and \(3\) modulo \(4\); both classes consist of odd
numbers.

In pre-extraction (i), the operation deletes a single part divisible by \(4\) and
records that same value. A part divisible by \(4\) lies in neither residue class
\(1\) nor \(3\); deleting it and recording it therefore changes neither \(r_1\) nor
\(r_3\) of the multiset of parts under consideration.

In iterative extraction (ii), each of the \(s-1\) surviving parts that precede the
chosen part \(4m\) is decreased from \(x\) to \(x-4\); by Remark~\ref{rem:wd}(b) we have \(x-4\ge 1\), so the part survives, and since \(x-4\equiv x\modulo 4\),
its residue is unchanged. The deleted part \(4m\) and the
recorded part \(4(s-1)+4m\) are both divisible by \(4\), hence lie outside the
classes counted by \(r_1,r_3\). Reordering permutes parts and does not affect
residue counts.

In insertion, inserting a recorded part \(4h\) increases each of the first \(h\)
parts from \(x\) to \(x+4\); again \(x+4\equiv x\modulo 4\), so residues are
preserved.

Thus no operation moves a part between the residue classes \(1\) and \(3\) modulo
\(4\), nor creates or destroys a part in either class. Hence
\(\bigl(r_1(\Phi(\alpha)),r_3(\Phi(\alpha))\bigr)=\bigl(r_1(\alpha),r_3(\alpha)\bigr)\).
Since \(\Phi:\Aone(n)\to\Atwo(n)\) is a bijection \cite[Section 3]{CGJL} and it
preserves the pair \((r_1,r_3)\), it restricts to a bijection between the
corresponding fibers, as claimed.
\end{proof}

The reader may wonder whether the two-phase structure of the extraction stage is really needed. It is.

\begin{remark}\label{rem:phases}
Both extraction phases are essential; neither yields a bijection $\Aone(n)\to\Atwo(n)$ on its own. If one skips pre-extraction (i) and applies only the iterative phase (ii), then \(\alpha=(5,4,1)\in\Aone(10)\) is sent to \((5,5)\), which has a repeated odd part and so lies outside \(\Atwo(10)\). If instead one keeps phase (i) but drops phase (ii), then \(\alpha=(6,4,1)\in\Aone(11)\) is returned unchanged (its part $4$ is not removable, since deleting it leaves the pair $(6,1)$ with gap $5$), and $(6,4,1)$ still contains a multiple of \(4\), so it lies outside \(\Atwo(11)\). The removability test of phase (i)---comparing, against the \(\Aone\) gap bound, the pair of neighbours left after a candidate deletion---is exactly the condition of Chen--Gao--Ji--Li \cite[Section 3, Step 1]{CGJL} specialized to \(N=2\); it guarantees that $\alpha^\ast$ satisfies both sets of conditions and the bound \(\delta_1\le 4\length(\alpha^\ast)\) on the largest part $\delta_1$ of $\delta$, under which the insertion phase returns a partition lying in \(\Atwo(n)\).
\end{remark}

\subsection{Map 3: from \texorpdfstring{$\Atwo(n)$}{A2(n)} to odd partitions}\label{sec:map3}

The third map dissolves the even parts of $\beta\in\Atwo(n)$: each part $2m$ with $m$ odd becomes two copies of $m$. What must be tracked is which odd values end up with odd multiplicity.

\begin{definition}[The map \(\psi^{-1}:\Atwo(n)\to\Oparts(n)\)]\label{def:psiinv}
For $\beta\in\Atwo(n)$ and odd $m\in\mathbb{Z}_+$, let $b_m$ be the multiplicity of $m$ in $\beta$ and $b_{2m}$ the multiplicity of $2m$ in $\beta$. Since $\beta \in \Atwo(n)$, no odd part repeats, so $b_m\in\{0,1\}$. We define $\mu=\psi^{-1}(\beta)$ by setting, for all odd $m$, $$\mult_\mu(m)=b_m+2b_{2m}.$$
By \textup{(A2.1)}, every part of $\beta$ is either odd or twice an odd number, so
\[
   |\mu|=\sum_{m\ \mathrm{odd}} m\,(b_m+2b_{2m})=|\beta|=n,
\]
and $\mu\in\Oparts(n)$.
This is the inverse of the Glaisher-type bijection $\psi:\Oparts(n)\to\Atwo(n)$ of Chen--Gao--Ji--Li \cite[Lemma 2.3]{CGJL}, under which an odd $m$ of multiplicity $k=2h+s$ in $\mu$ (with $s\in\{0,1\}$) contributes $h$ copies of $2m$ and $s$ copies of $m$. One checks directly that \(\psi\) and \(\psi^{-1}\) are mutually inverse, so \(\psi^{-1}\) is a bijection.
\end{definition}

\begin{example}\label{ex:psiinv}
Take $\beta = (10,3,2,1)$. For $m=5$: $b_5=0$ and $b_{10}=1$, so $\mult_\mu(5)=0+2\cdot1=2$.
For $m=3$: $b_3=1$ and $b_6=0$, so $\mult_\mu(3)=1$.
For $m=1$: $b_1=1$ and $b_2=1$, so $\mult_\mu(1)=1+2=3$.
Hence $$\psi^{-1}(10,3,2,1)=(5,5,3,1,1,1).$$
\end{example}

The lemma we need is that the parity of $\mult_\mu(m)$ remembers exactly whether $m$ was a part of $\beta$.

\begin{lemma}[Odd multiplicities after \(\psi^{-1}\)]\label{lem:psiinv}
Let $\mu=\psi^{-1}(\beta)$ and let $m$ be odd. Then $\mult_\mu(m)$ is odd if and only if $m$ is a part of $\beta$. In particular, the odd values occurring with odd multiplicity in $\mu$ are precisely the odd parts of $\beta$; the value $m$, and hence its residue modulo $4$, is the same on both sides.
\end{lemma}

\begin{proof}
By Definition~\ref{def:psiinv}, $\mult_\mu(m)=b_m+2b_{2m}\equiv b_m\modulo 2$. As $b_m\in\{0,1\}$, $\mult_\mu(m)$ is odd if and only if $b_m=1$, i.e.\ if and only if $m$ is a part of $\beta$.
\end{proof}

\subsection{Map 4: Glaisher's bijection}\label{sec:map4}
The final map is Glaisher's classical bijection from odd partitions to strict partitions: multiplicities are written in binary.

\begin{definition}[The map \(G:\Oparts(n)\to\D(n)\)]\label{def:G}
Let \(\mu\in\Oparts(n)\). For each odd \(m\), write the multiplicity of \(m\) in
binary,
\[
   \mult_\mu(m)=\sum_{c\ge 0}\epsilon_{m,c}\,2^c,
   \qquad \epsilon_{m,c}\in\{0,1\},
\]
and replace the \(\mult_\mu(m)\) copies of \(m\) by the parts \(2^c m\)
for those \(c\) with \(\epsilon_{m,c}=1\). These parts are pairwise
distinct: for a fixed $m$ the exponents $c$ differ, and parts arising from different odd values $m$ differ because every positive integer factors uniquely as $2^c m$ with $m$ odd. The result is therefore a strict partition \(G(\mu)\); see \cite{Glaisher}. The inverse of $G$ factors each part of a strict partition uniquely as \(2^c m\) with \(m\) odd and replaces it by \(2^c\) copies of
\(m\). (The map is recalled in
\cite[Section 1]{CGJL}.)
\end{definition}

\begin{example}\label{ex:G}
Take $\mu = (5,5,3,1,1,1)$. Then $\mult_\mu(5)=2=2^1$, contributing the part $2^1\cdot 5=10$; $\mult_\mu(3)=1=2^0$, contributing the part $3$; and $\mult_\mu(1)=3=2^1+2^0$, contributing the parts $2$ and $1$. Hence
$$G(5,5,3,1,1,1)=(10,3,2,1).$$
\end{example}

Mirroring Lemma~\ref{lem:psiinv}, the odd parts of $G(\mu)$ remember exactly which values had odd multiplicity in $\mu$.

\begin{lemma}[Odd parts after Glaisher]\label{lem:G}
Let $\sigma = G(\mu)$ and let $m$ be odd. Then $m$ is a part of $\sigma$ if and only if $\mult_\mu(m)$ is odd, and every odd part of $\sigma$ arises this way. In particular, the odd parts of $\sigma$ are precisely the odd values occurring with odd multiplicity in $\mu$; the value $m$, and hence its residue modulo $4$, is the same on both sides.
\end{lemma}

\begin{proof}
The parts of $\sigma$ produced from the copies of $m$ are the numbers $2^c m$ with $\epsilon_{m,c}=1$. Among these, only $2^0m=m$ is odd, and it occurs exactly when $\epsilon_{m,0}=1$, i.e.\ when $\mult_\mu(m)$ is odd. Conversely, every odd part of $\sigma$ has the form $2^0m=m$ for some odd value $m$ of $\mu$, since every part of $\sigma$ is $2^cm$ for some part value $m$ of $\mu$, all of which are odd. Therefore the odd parts of $\sigma$ are in value-preserving bijection with the odd-multiplicity values of $\mu$.
\end{proof}

\subsection{The bijection and its proof}\label{sec:proof}
We now assemble the four maps and prove the refined form of Theorem~\ref{thm:BU}. For fixed \(n,i,j\ge0\) set
\[
   \D_{i,j}(n)=\{\lambda\in\D(n): e_1(\lambda)=i,\ e_3(\lambda)=j\},
   \qquad
   \D'_{i,j}(n)=\{\sigma\in\D(n): r_1(\sigma)=i,\ r_3(\sigma)=j\},
\]
the fibers of the statistic maps $\lambda\mapsto(e_1(\lambda),e_3(\lambda))$ and $\sigma\mapsto(r_1(\sigma),r_3(\sigma))$ on $\D(n)$,
and recall \(F=G\circ\psi^{-1}\circ\Phi\circ\varphi\). Concretely, for
\(\lambda\in\D(n)\) one computes
\[
   \lambda \xmapsto{\ \varphi\ }
   \alpha \xmapsto{\ \Phi\ }
   \beta \xmapsto{\ \psi^{-1}\ }
   \mu \xmapsto{\ G\ }
   \sigma,
   \qquad F(\lambda)=\sigma .
\]

\begin{theorem}\label{thm:main}
For all \(n,i,j\ge 0\), the map \(F=G\circ\psi^{-1}\circ\Phi\circ\varphi\)
restricts to a bijection
\[
   F:\D_{i,j}(n)\longrightarrow \D'_{i,j}(n) .
\]
In particular, \(|\D_{i,j}(n)|=|\D'_{i,j}(n)|\), which is Theorem~\ref{thm:BU}.
\end{theorem}

\begin{proof}
Every factor of $F$ is a bijection between the indicated sets: \(\varphi:\D(n)\to\Aone(n)\) by \cite[Lemma 2.2]{CGJL}, \(\Phi:\Aone(n)\to\Atwo(n)\) by \cite[Section 3]{CGJL}, \(\psi^{-1}:\Atwo(n)\to\Oparts(n)\) by \cite[Lemma 2.3]{CGJL}, and \(G:\Oparts(n)\to\D(n)\) by \cite{Glaisher}. Their composition $F$ is therefore a bijection of $\D(n)$ onto itself. It remains to show that $F$ carries each fiber $\D_{i,j}(n)$ onto the corresponding fiber $\D'_{i,j}(n)$.

First we show \(F(\D_{i,j}(n))\subseteq\D'_{i,j}(n)\). Let \(\lambda\in\D_{i,j}(n)\) and \(\alpha=\varphi(\lambda)\), \(\beta=\Phi(\alpha)\), \(\mu=\psi^{-1}(\beta)\), \(\sigma=G(\mu)=F(\lambda)\). By
Lemma~\ref{lem:phi},
\[
   r_1(\alpha)=e_1(\lambda)=i,\qquad r_3(\alpha)=e_3(\lambda)=j .
\]
By Lemma~\ref{lem:Phi},
\[
   r_1(\beta)=r_1(\alpha)=i,\qquad r_3(\beta)=r_3(\alpha)=j .
\]
As \(\beta\in\Atwo(n)\) has no part divisible by \(4\), every part of
\(\beta\) is congruent to \(1\), \(2\), or \(3\) modulo \(4\); the parts counted by
\(r_1(\beta)\) and \(r_3(\beta)\) are therefore exactly the odd parts of \(\beta\)
in the residue classes \(1\) and \(3\) modulo \(4\), respectively. By
Lemma~\ref{lem:psiinv}, these odd parts of \(\beta\) are, value for value, the odd values occurring with odd multiplicity in \(\mu\), with residue modulo \(4\)
preserved. By Lemma~\ref{lem:G}, those odd-multiplicity values of \(\mu\) are,
value for value and residue for residue, the odd parts of \(\sigma\). Finally,
every odd part of \(\sigma\) lies in the residue class \(1\) or \(3\) modulo \(4\),
and no even part of \(\sigma\) can lie in either; hence \(r_1(\sigma)\) and
\(r_3(\sigma)\) count exactly the odd parts of \(\sigma\) in the respective
classes. Combining the three correspondences,
\[
   r_1(\sigma)=r_1(\beta)=i,\qquad r_3(\sigma)=r_3(\beta)=j ,
\]
so \(\sigma\in\D'_{i,j}(n)\). This proves \(F(\D_{i,j}(n))\subseteq\D'_{i,j}(n)\) for every pair $(i,j)$.

Now the sets \(\D_{i,j}(n)\), as \((i,j)\) ranges over pairs of nonnegative integers, are pairwise disjoint with union \(\D(n)\), since each \(\lambda\) determines its pair \((e_1(\lambda),e_3(\lambda))\); likewise the sets \(\D'_{i,j}(n)\) are pairwise disjoint with union \(\D(n)\). Because \(F\) is a bijection of \(\D(n)\),
\[
   \D(n)=F(\D(n))=\bigsqcup_{i,j\ge 0} F(\D_{i,j}(n))
   \subseteq \bigsqcup_{i,j\ge 0} \D'_{i,j}(n)=\D(n),
\]
so the containment in the middle is an equality of disjoint unions. Since each \(F(\D_{i,j}(n))\) is contained in \(\D'_{i,j}(n)\) and the latter sets are pairwise disjoint, this forces
\(F(\D_{i,j}(n))=\D'_{i,j}(n)\) for every \((i,j)\). Hence the restriction \(F:\D_{i,j}(n)\to\D'_{i,j}(n)\) is a bijection, and equating cardinalities
yields Theorem~\ref{thm:BU}.
\end{proof}

\begin{remark}
The inverse bijection is the composition of the inverses of the four factors,
\[
   F^{-1}=\varphi^{-1}\circ\Phi^{-1}\circ\psi\circ G^{-1}
   :\D(n)\longrightarrow\D(n),
\]
each of which is described explicitly in the corresponding definition. Thus $F$ and $F^{-1}$ are both computable by elementary manipulations of diagrams and multiplicities.
\end{remark}

\section{Examples and verification}\label{sec:numerics}

This section makes the bijection concrete. Section~\ref{sec:example} traces a single strict partition through all four maps, and Section~\ref{sec:table} tabulates $F$ on an entire fiber, exhibiting the bijection of Theorem~\ref{thm:main} explicitly for $n=16$ and $i=j=1$.

\subsection{A worked example}\label{sec:example}
Take \(\lambda=(7,6,2,1)\in\D(16)\), which has odd parts $7$ and $1$ in positions $1$ and $4$, respectively, so that $(e_1,e_3)=(1,1)$.
\paragraph{Map 1.}
The \(2\)-modular diagram of $\lambda$ is
\[
\begin{array}{@{}cccc@{}}
2 & 2 & 2 & 1\\
2 & 2 & 2 & \\
2 & & & \\
1 & & &
\end{array}
\]
with rows $2\,2\,2\,1$, $2\,2\,2$, $2$, $1$ recording the parts $7,6,2,1$. Reading the column sums,
\[
   \alpha=\varphi(7,6,2,1)=(7,\ 4,\ 4,\ 1),
\]
and one checks that $(7,4,4,1)\in\Aone(16)$: only the even part $4$ repeats, the gaps are $3,0,3$, and the smallest part is $1$.

\paragraph{Map 2.}
This is precisely Example~\ref{ex:Phi2}: pre-extraction removes the upper $4$ and records $4$; iterative extraction turns $(7,4,1)$ into $\alpha^\ast=(3,1)$ and records $8$; and inserting $\delta=(8,4)$ gives
\[
   \beta=\Phi(7,4,4,1)=(11,5).
\]

\paragraph{Map 3.}
As all parts of $\beta=(11,5)$ are odd, $\psi^{-1}$ fixes it: \(\mu=(11,5)\).

\paragraph{Map 4.}
As all multiplicities in \(\mu\) are \(1\), Glaisher's map fixes it:
\[
   F(7,6,2,1)=G(11,5)=(11,5).
\]
\paragraph{Statistic transfer.} The output has one part congruent to \(1\modulo 4\) (namely \(5\)) and one congruent to \(3\modulo 4\) (namely \(11\)), in agreement with \((e_1,e_3)=(1,1)\).

\subsection{Verification table}\label{sec:table}

Take $n=16$ and $i=j=1$. The set $\D_{1,1}(16)$ of strict partitions of $16$ with $e_1=e_3=1$ has $15$ elements, listed in the first column below; the remaining columns give the successive images $\alpha$, $\beta$, $\mu$, and $F(\lambda)$. Each output is strict with \(r_1(\sigma)=r_3(\sigma)=1\), and the outputs are pairwise distinct; thus the table exhibits the bijection \(F:\D_{1,1}(16)\to\D'_{1,1}(16)\) explicitly. The row $\lambda=(9,5,2)$ strings together Examples~\ref{ex:Phi1}, \ref{ex:psiinv}, and~\ref{ex:G}, and the row $\lambda=(7,6,2,1)$ is the worked example of Section~\ref{sec:example}.

\setlength{\tabcolsep}{2pt}
{\footnotesize
\begin{longtable}{@{}>{\raggedright\arraybackslash}p{0.13\linewidth}>{\raggedright\arraybackslash}p{0.23\linewidth}>{\raggedright\arraybackslash}p{0.17\linewidth}>{\raggedright\arraybackslash}p{0.28\linewidth}>{\raggedright\arraybackslash}p{0.13\linewidth}@{}}
\toprule
\(\lambda\) & \(\alpha=\varphi(\lambda)\) & \(\beta=\Phi(\alpha)\) & \(\mu=\psi^{-1}(\beta)\) & \(F(\lambda)=G(\mu)\) \\
\midrule
\endhead
(15, 1) & (3, 2, 2, 2, 2, 2, 2, 1) & (3, 2, 2, 2, 2, 2, 2, 1) & (3, 1, 1, 1, 1, 1, 1, 1, 1, 1, 1, 1, 1, 1) & (8, 4, 3, 1) \\
(13, 3) & (4, 3, 2, 2, 2, 2, 1) & (7, 2, 2, 2, 2, 1) & (7, 1, 1, 1, 1, 1, 1, 1, 1, 1) & (8, 7, 1) \\
(12, 3, 1) & (5, 3, 2, 2, 2, 2) & (5, 3, 2, 2, 2, 2) & (5, 3, 1, 1, 1, 1, 1, 1, 1, 1) & (8, 5, 3) \\
(11, 5) & (4, 4, 3, 2, 2, 1) & (11, 2, 2, 1) & (11, 1, 1, 1, 1, 1) & (11, 4, 1) \\
(11, 3, 2) & (6, 3, 2, 2, 2, 1) & (6, 3, 2, 2, 2, 1) & (3, 3, 3, 1, 1, 1, 1, 1, 1, 1) & (6, 4, 3, 2, 1) \\
(10, 5, 1) & (5, 4, 3, 2, 2) & (9, 3, 2, 2) & (9, 3, 1, 1, 1, 1) & (9, 4, 3) \\
(9, 7) & (4, 4, 4, 3, 1) & (15, 1) & (15, 1) & (15, 1) \\
(9, 5, 2) & (6, 4, 3, 2, 1) & (10, 3, 2, 1) & (5, 5, 3, 1, 1, 1) & (10, 3, 2, 1) \\
(9, 4, 2, 1) & (7, 4, 2, 2, 1) & (7, 6, 2, 1) & (7, 3, 3, 1, 1, 1) & (7, 6, 2, 1) \\
(8, 7, 1) & (5, 4, 4, 3) & (13, 3) & (13, 3) & (13, 3) \\
(8, 5, 3) & (6, 5, 3, 2) & (6, 5, 3, 2) & (5, 3, 3, 3, 1, 1) & (6, 5, 3, 2) \\
(8, 4, 3, 1) & (7, 5, 2, 2) & (7, 5, 2, 2) & (7, 5, 1, 1, 1, 1) & (7, 5, 4) \\
(7, 6, 2, 1) & (7, 4, 4, 1) & (11, 5) & (11, 5) & (11, 5) \\
(7, 5, 4) & (6, 6, 3, 1) & (6, 6, 3, 1) & (3, 3, 3, 3, 3, 1) & (12, 3, 1) \\
(6, 5, 3, 2) & (8, 5, 3) & (9, 7) & (9, 7) & (9, 7) \\
\bottomrule
\end{longtable}
}

\section*{Acknowledgements}
The authors thank Aritram Dhar for his helpful comments.

\section*{Process}\label{sec:AI}
The mathematics of this paper, the four-map bijection and its statistic
bookkeeping, was worked out by the human authors. AxiomProver was asked to
formalize and verify that argument, not to discover it. 
The formal proofs provided were verified using Lean~4.31.0 \cite{Lean}.
Compatibility with earlier or later versions is not guaranteed due to the
evolving nature of the Lean 4 compiler and its core libraries.
The relevant files are all posted in the following repository:
\begin{center}
  \url{https://github.com/AxiomMath/BerkovichUncu}
\end{center}
The formalization consisted of \texttt{input/}, which contained the following:
\begin{itemize}
\item \texttt{task.md}: a natural-language description of the problem to be proven and formalized.
\item \texttt{main-5.tex}: a draft of the paper.
\end{itemize}
Given these files,
AxiomProver autonomously produced the following output files, collected under \texttt{BerkovichUncu/}:
\begin{itemize}
 \item \texttt{problem.lean}: a translation of the problem statement into Lean.
 \item \texttt{solution.lean}: the formal, machine-checked solution in Lean.
\end{itemize}
The translation of a paper had to assume certain facts, which are described in the
\texttt{README.md} of the repository. Each is either quoted above from
\cite{CGJL} and \cite{Glaisher} or established above together with
the definition of the map concerned; none is a blind assumption.

\section*{Declaration of generative AI and AI-assisted technologies in the manuscript preparation process}
As described in the preceding Section,
AxiomProver (an AI tool under development)
was used to produce a formal verification of the problem.
The paper was written without AI.


\begin{thebibliography}{99}

\bibitem{BU}
A. Berkovich and A. K. Uncu,
\emph{On partitions with fixed number of even-indexed and odd-indexed odd parts},
J. Number Theory \textbf{167} (2016), 7--30.
\href{https://doi.org/10.1016/j.jnt.2016.02.031}{doi:10.1016/j.jnt.2016.02.031}.

\bibitem{Bess1991}
C. Bessenrodt,
\emph{A combinatorial proof of a refinement of the Andrews--Olsson partition
identity},
European J. Combin. \textbf{12} (1991), 271--276.

\bibitem{Bess1995}
C. Bessenrodt,
\emph{Generalizations of the Andrews--Olsson partition identity},
Discrete Math. \textbf{141} (1995), 11--22.

\bibitem{CGJL}
W. Y. C. Chen, H. Y. Gao, K. Q. Ji, and M. Y. X. Li,
\emph{A unification of two refinements of Euler's partition theorem},
Ramanujan J. \textbf{23} (2010), 137--149.
\href{https://doi.org/10.1007/s11139-008-9156-7}{doi:10.1007/s11139-008-9156-7}.

\bibitem{Lean}
L.~de~Moura, S.~Kong, J.~Avigad, F.~van~Doorn, and J.~von~Raumer,
The {L}ean theorem prover (system description),
in \emph{Automated Deduction -- CADE-25},
Lecture Notes in Computer Science~9195, Springer, 2015, 378--388.

\bibitem{DM}
A. Dhar and A. Mukhopadhyay,
\emph{A bijective proof of an identity of Berkovich and Uncu},
S\'em. Lothar. Combin. \textbf{90} (2026), Article B90a, 16 pp.

\bibitem{FZ}
S. Fu and J. Zeng,
\emph{A unifying combinatorial approach to refined little G\"ollnitz and
Capparelli's companion identities},
Adv. in Appl. Math. \textbf{98} (2018), 127--154.
\href{https://doi.org/10.1016/j.aam.2018.03.005}{doi:10.1016/j.aam.2018.03.005}.

\bibitem{Glaisher}
J. W. L. Glaisher,
\emph{A theorem in partitions},
Messenger Math. \textbf{12} (1883), 158--170.

\end{thebibliography}
\end{document}